\documentclass[11pt,a4paper]{article}

\usepackage[theorem/counter/prefix = section]{mac-bundle}
\usepackage{bookmark} 

\DeclareTheorem[plain]{question}{Question}

\newwrapper {macros}
\newwrapper {citetree}

\begin{macros}

\newcommand  \slice {\mathrel{/}}
\newcommand \G {\mathbb{G}}

\newfunction \ev       {ev}
\newfunction \id       {id}
\newfunction \filtered {filt}
\newfunction \Fun      {Fun}
\newfunction \Ind      {Ind}
\newfunction \Map      {Map}
\newfunction \Sheaf    {Sh}
\newfunction \Presheaf {PSh}
\newfunction \Mod      {Mod}
\newfunction \target   {tar}
\newfunction \source   {src}
\newfunction \End      {End}
\newfunction \RFib     {RFib}
\newfunction \forget   {fgt}
\newfunction \coCart   {coCart}
\newfunction \Object   {Ob}
\newfunction \Nat      {Nat}

\newfunction \Alg      {Alg}
\newfunction \Aut      {Aut}
\newfunction \GL       {GL}
\newconstant \Kan {Kan}
\newfunction \Spec {Spec}

\newfunction \PrL {Pr^L}
\newconstant \FP {FP}

\newconstant \point  {1}
\newconstant \Cat    {Cat}
\newconstant \Topos  {Topos}
\newconstant \Space  {Spc}
\newconstant \Set    {Set}
\newconstant \Sketch {Sketch}

\newconstant \Idem {Idem}

\newconstant \Scheme {Sch}
\newcommand \A {\mathbb{A}}

\newfunction \canonical {can}
\newfunction \RP {RP}

\newfunction \Rel {Rel}
\newfunction \Kon {Kon}
\newfunction \Yoneda {yo}
\newfunction \Fib {Fib}

\newconstant \Sm {Sm}

\newfunction \mori {Morita}

\newfunction \Atlas {At}
\newfunction \DM {Stk^{DM}}
\newfunction \Stack {Stk}
\newfunction \Segal {Seg}

\newfunction \symb {\Sigma}

\makeatletter

\newcommand \pathdef [3] {%
  \expandafter\def\csname @#1\endcsname{\cite[#3]{#2}}%
}

\newcommand \pathcite [1] {%
  \ifcsdef{@#1}{%
    \csname @#1\endcsname%
  }{%
    \PackageWarning{paths}{Path "#1" not found}%
  }%
}
\makeatother

\end{macros}

\begin{citetree}
\pathdef {cis/localization/bousfield/criterion} {Cisinski:HigherCats} {Prop.~7.11.2}
\pathdef {htt/bousfield-localization/map-out} {HTT} {Prop.~5.5.4.20}
\end{citetree}

\title{Fibrations and Koszul duality in locally Cartesian localisations}
\author{Andrew W. Macpherson}
\begin{document}

\maketitle

\begin{abstract}

  I show that any locally Cartesian left localisation of a presentable $\infty$-category admits a right proper model structure in which all morphisms are cofibrations, and obtain a Koszul duality classification of its fibrations.
  By a simple criterion in terms of generators for a localisation to be locally Cartesian, this applies to any nullification functor. 
  In particular, it includes examples with non-trivial `homotopical content'.
  
  We further describe, and provide examples from, the set of fibrations in three contexts: the higher categorical Thomason model structure of Mazel-Gee, where fibrations are local systems; Morel-Voevodsky $\A^1$-localisation, where they are a higher analogue of $\A^1$-covering spaces; and the Quillen plus construction, where they are related to loop space modules trivialised over the universal acyclic extension.  
  
\end{abstract}

\tableofcontents
\section{Introduction}

The concept of model $\infty$-categores was introduced in a series of papers starting with \cite{mazelgee2015modela} to address the need for structures to get control over localisations of $\infty$-categories.
The material of \emph{op.~cit}.~allows us to make arguments about model $\infty$-categories that are largely identical to their 1-category cousins; yet describing interesting practical examples is often much easier, essentially because the subtleties of, say, localising from presheaves of simplicial sets to presheaves of spaces, has already been dispensed with.

In this paper we concern ourselves with a class of model structures on $\infty$-categories which are relatively trivial to construct and work with: model structures that present a colimit-preserving localisation --- so weak equivalences are stable under colimit --- and in which all morphisms are cofibrations.

\begin{theorem}[Thm.~\ref{bousfield/model/exists}]

  Let $C$ be a presentable $\infty$-category, $L:C\rightarrow D$ a left Bousfield localisation (i.e.~a functor admitting a fully faithful right adjoint).
  Then there is a left proper cofibrantly generated model structure on $C$ --- the \emph{left localisation model structure} --- such that:
  \begin{itemize}
    \item[(C)] All morphisms of $C$ are cofibrations;
    \item[(W)] Weak equivalences are exactly the $L$-equivalences.
    \item[(F)] Fibrations are right orthogonal to $L$-equivalences. That is, they satisfy a \emph{unique} right lifting property.
  \end{itemize}
  All morphisms in $D$ are fibrations, and conversely, any fibration with codomain in $D$ has domain in $D$. In particular, $D$ is the category of fibrant objects.
  
\end{theorem}

In this paper, we focus on the case of \emph{locally Cartesian} localisations \cite{gepner2017univalence}.
A generating set criterion means that it is easy to construct examples in this context; on the other hand, a criterion to be a fibration helps us to get a grip on this class of maps.

\begin{theorem}[Props.~\ref{locally-cartesian/criterion}, \ref{locally-cartesian/generator}, \ref{locally-cartesian/unit-generates}]

  Let $C$ be an $\infty$-topos, $L:C\rightarrow D$ an accessible localisation.
  The following conditions are equivalent:
  \begin{enumerate}
    \item 
      $L$ is locally Cartesian;
      
    \item 
      the left localisation model structure of $L$ is right proper;
      
    \item 
      (if every arrow in $D$ is exponentiable in $C$) $L$ is generated as a Bousfield localisation by a set of morphisms stable for base change.
  \end{enumerate}
  In this case,
      a morphism $X\rightarrow Y$ in $C$ is an $L$-fibration if and only if the square
      \[
        \begin{tikzcd}
          X \ar[r] \ar[d] & LX \ar[d] \\
          Y \ar[r] & LY 
        \end{tikzcd}
      \]
      is a pullback.
  
\end{theorem}

The condition to be a fibration for a proper left localisation model structure is very restrictive. 
For example, it is easy to see that `fibres of fibrations are local,' in the sense that if $f:X\rightarrow Y$ is a fibration and $S$ is a local object mapping to $Y$, then $S\times_YX\cong S\times_{Y^+}X^+$. 
In particular, this fibre product is local.
On the other hand, this restrictiveness allows us to formulate a kind of `classification' result for fibrations:

\begin{theorem}[Thm.~\ref{koszul/classification}] \label{main/koszul}

  Let $L:C\rightarrow D$ be a locally Cartesian localisation, and suppose that either $C$ or $D$ is an $\infty$-topos.
  Then for any connected, pointed object $(X,x)$ of $D$ we have a Koszul duality equivalence
  \[
    \mathrm{Fib}\slice X \cong \Mod_{\Omega_xLX}(D)
  \]
  that intertwines the fibre functor $x^*:\Fib\slice X\rightarrow C$ with the forgetful functor $\Mod_{\Omega_xLX}(D)\rightarrow D\subseteq C$.
  
\end{theorem}

\paragraph{Nullification}

As well as the already rich family of examples coming from $\infty$-topos theory, many important examples of locally Cartesian localisations can be constructed as \emph{nullifications} (Definition \ref{nullification}):

\begin{corollary}[of Thm.~\ref{main/koszul}]

  Let $L:C\rightarrow D$ be a nullification functor, and suppose that every morphism in $D$ is exponentiable in $C$ (for example, if $C$ is an $\infty$-topos).
  Then $L$ is locally Cartesian.
  
\end{corollary}

Nullifications can have highly nontrivial `homotopical content' in the sense that the localisation of truncated objects can have interesting higher homotopy groups. (Left exact localisations, on the other hand, preserve truncatedness.) 
For example:
\begin{itemize}

  \item 
    Groupoid completion $|-|:\Cat_\infty\rightarrow\Space$. Every space arises as the completion of a preorder (i.e.~a 0-truncated object).
    
  \item 
    Morel-Voevodsky $\mathbb{A}^1$-localisation. Localising $\mathbb{P}^\infty$ yields an object with $\pi_1(L_{\mathbb{A}^1}\mathbb{P}^\infty) \cong \mathbb{G}_m$ \cite[Prop.~4.3.8]{morel19991}.
    
  \item 
    The Quillen plus construction. Localising a 1-truncated object $\Z\times \mathrm{BGL}(R)$ yields algebraic $K$-theory.
  
\end{itemize}
One does not have to look far to find further examples, such as more general nullifications of spaces \cite{farjoun2006cellular} or extended theories of motives such as `motives with modulus' \cite{kahn2020motives}.

In the remainder of the paper, we study the class of fibrations in three settings, exploring the interpretation of the Koszul duality equivalence and constructing examples of fibrations from univalent families \cite{gepner2017univalence}.

\paragraph{Thomason model structure}
The first example is the $\infty$-categorical analogue of the Thomason model structure \cite{thomason1980cat, mazelgee2015grothendieck}, which was the original motivation for this paper.
This is much easier to understand than the classical Thomason structure: groupoid completion of $\infty$-categories is a nullification (of $\Delta^1$), and hence is presented by a proper left localisation model structure on $\Cat_\infty$ .
The fibrations are precisely the \emph{local systems}, that is functors which are (categorically equivalent to) both left and right fibrations. (Meanwhile, the fibrations and cofibrations in the classical Thomason model structure are rather hard to describe explicitly \cite{bruckner2016cofibrant}.)

\begin{theorem}[\S\ref{thomason/}]

  The localisation $|-|:\Cat_\infty\rightarrow\Space$ is presented by a proper left localisation `Thomason' model structure on $\Cat_\infty$ in which a functor is a fibration if and only if it is classified by a local system.
  
  The category of Thomason fibrations over a connected $\infty$-category $C$ is equivalent to $\Mod_{\Omega_c|C|}$ for any $c:C$.

\end{theorem}        

\paragraph{Motivic spaces}
In the case of the Morel-Voevodsky $\A^1$-localisation of Nisnevich sheaves, our theory of fibrations is an $\infty$-analogue of the notion of $\A^1$-\emph{cover} studied in \cite{asok2009a1} and used to compute $\A^1$-homotopy groups of toric varieties.
Koszul duality gives a higher `geometric' interpretation of modules over the $\A^1$-loop space analogous to \cite[Thm.~3.16]{asok2009a1}.
Via univalent families, we can construct some classes of explicit examples.

\begin{theorem}[\S\ref{motive/}]

  Morel-Voevodsky $\A^1$-localisation $L_{\A^1}:\Sheaf_\infty(\Sm_k,\mathrm{Nis})\rightarrow M_k$ is presented by a proper left localisation model structure on $\Sheaf_\infty(\Sm_k)$.
  
  The associated class of $\A^1$-fibrations is classified via Koszul duality:
  \[
    \A^1\Fib\slice X \cong \Mod_{\Omega_xL_{\A^1}X} \left(\Sheaf^{\A^1}(\Sm, \mathrm{Nis}) \right)
  \]
  for any $\A^1$-connected Nisnevich sheaf $X$.
  Moreover:
  \begin{enumerate}
  
    \item
      Any map admitting a lift to an $\A^1$-cover between Nisnevich-fibrant simplicial presheaves is an $\A^1$-fibration.
      The converse is true for maps betweeen Nisnevich sheaves of sets.
  
    \item
      If the $k$ is a number field and Lang's conjectures hold, then any proper, smooth family of hyperbolic varieties is an $\A^1$-fibration.
      
    \item
      An $A$-banded $n$-gerbe is an $\A^1$-fibration for any $n\geq 0$ and $A$ either $\G_m$ or finite \'etale with torsion prime to the characteristic of the ground field.
  
  \end{enumerate}
\end{theorem}

\paragraph{Plus construction}
A special feature of the plus construction among nullification functors is that it can be built by coning off a single acyclic space.
Via descent theory and Koszul duality, this corresponds to trivialising the action of the loop space over the universal acyclic extension thereof. 

Studying pullbacks preserved by the plus construction invites a comparison with excision in algebraic $K$-theory; unfortunately, we find the domain of overlap with our theory of fibrations is trivial.

\begin{theorem}[\S\ref{plus/}]

  The Quillen plus construction $(-)^+:\Space\rightarrow\Space_{solv}$, where $\Space_{solv}\subseteq\Space$ is the fixed set of the plus construction comprising the spaces with pro-solvable fundamental group, is presented by a proper left localisation model structure.
  
  The category of plus-fibrations over a connected pointed space $(X,x)$ is equivalent to the category of $\Omega_xX$-modules equipped with a trivialisation over the universal acyclic extension of $\Omega_xX$.
  Moreover, we have the following examples and non-examples of plus-fibrations:
  \begin{enumerate}
    \item
      Any map that can be expressed as a composite of spherical fibrations is a plus-fibration.
      In particular, an acyclic $\infty$-group cannot act nontrivially on a sphere.
      
    \item
      The universal $n$-torus fibration for $n>1$ is not a plus-fibration.
      
    \item
      A surjective algebra homomorphism $\phi:A\rightarrow A/I$ induces a plus-fibration $B\GL(\phi):B\GL(A)\rightarrow B\GL(A/I)$ if and only if $I=0$.
  \end{enumerate}

\end{theorem}

\subsection*{Discussion}

  From the perspective of classical model category theory, model structures presenting a colimit-preserving localisation would be considered utterly trivial.
  Yet in $\infty$-category theory, they already include nearly all the interesting examples I can think of.
  Undoubtedly, with a little extra effort, a much more general existence result for Bousfield localisations can be found along the lines of \cite[Thm.~4.7]{barwick2010left}, but I don't have any application in mind for this.
  %
  
  Since we are only considering localisation functors that preserve colimits, the utility of these model structures, if any, is in computing homotopy \emph{limits}.
  However, study of the examples shows that fibrations for left localisation model structures are not very abundant.
  In each case there are certainly more examples of maps whose pullbacks are either preserved by localisation or have an error term that is more or less `computable'.
  Perhaps in these cases, we would do better to look for a weaker structure such as that of an $\infty$-category of fibrant objects in the sense of \cite[Def.~7.4.12]{HCHA}. 
  Anyway, I leave the search for such structures for another day.
  
\subsection*{Previous work}

This work is a response to a remark made in \cite{mazelgee2015modela}.
The contents of \S\ref{locally-cartesian/} were greatly inspired by \cite[\S1.5, \S4.3, and \S6.1]{cisinski2006prefaisceaux}.
Some properties of locally Cartesian localisations were studied in \cite{gepner2017univalence}, \cite[\S3]{Hoyois_2017}.

\subsection*{Acknowledgement}

Thanks to Marc Hoyois for helping me to understand a key fact about fibre sequences in $\infty$-topoi, Ryomei Iwasa for introducing me to excision in K-theory, to Shane Kelly for helping me with the motives literature, and to Aaron Mazel-Gee for telling me about model $\infty$-categories in the first place.

\subsection*{Conventions}

We use the Joyal-Lurie theory of quasi-categories as our foundation. 
As in \cite{HCHA} and much of the classical category theory literature, the slice category of a category $C$ over (resp.~under) an object $X$ is denoted $C\slice X$ (resp.~$X\slice C$). 
The `sandwich' category, which appears as the space of solutions to a lifting problem, is written $X\slice C\slice Y\defeq X\slice C\times_CC\slice Y$.
From this point on we mostly drop the prefix `$\infty$-' and refer to (the objects modelled by) quasi-categories as just `categories.' 
Similarly, we drop the subscript from $\Cat_\infty$.
Categories of sets, spaces, and categories, presheaves, presentability, and `smallness' are all assessed against a universe which is fixed throughout.

\section{Localisation}

In this section we review the basics of localisations and establish a simple existence result for Bousfield localisations of model structures on $\infty$-categories.

\begin{para}[Bousfield localisation]
\label{bousfield}

  Recall \pathcite{cis/localization/bousfield/criterion} that a \emph{Bousfield localisation} of an $\infty$-category $C$ is a functor $L:C\rightarrow D$ which admits an accessible right adjoint $R$ satisfying the equivalent conditions:
  \begin{itemize}

    \item $R$ is fully faithful (and therefore embeds $D$ as a reflective subcategory of $C$);
    
    \item $L$ is a localisation of $C$;

    \item $L$ is a localisation at the set $W_L$ of morphisms that become invertible in $C$.

  \end{itemize}
  In this case, the elements of $W_L$ are called $L$-\emph{equivalences}.
  The essential image of $D$ under $R$ consists exactly of the $W$-\emph{local} objects.
  
  A Bousfield localisation $L:C\rightarrow D$ is said to be \emph{generated} by a set of morphisms $W\subseteq C$ if $D$ is the category of $W$-local objects of $C$.
  (Note that this does not imply that $W$ generates $L$ as an ordinary localisation.)
  
\end{para}

\begin{lemma}[Bousfield localisations are generated by counits] \label{bousfield/generation}

  Let $L:C\rightarrow D$ be a Bousfield localisation with right adjoint $R$.
  Let $W\subseteq C^{\Delta^1}$ be set of maps of the form $X\rightarrow RLX$.
  Then $L$ is generated as a localisation by $W$.
  
\end{lemma}
\begin{proof}

  It is equivalent to show that $W$-local presheaves descend to $D$. 
  But $W$-locality of a presheaf $F$ says precisely that $F(e_X):F(RLX)\tilde\rightarrow F(X)$ for all $X:C$, that is, $F\cong L^*R^*F$. \qedhere
  
\end{proof}

\begin{definition}
 
  A set of arrows $K$ in a category $C$ is \emph{strongly saturated} if every morphism in $C$ which becomes invertible in $C[K^{-1}]$ is already in $K$.
  
\end{definition}

\begin{para}[Model structures]

  We freely use certain notions and basic properties of (weak) factorization systems and sets of morphisms satisfying lifting properties which are well-known from classical model category theory (and somewhat reviewed in \cite[\S1]{mazelgee2015modela}.
  
  To every set of morphisms $K\subseteq C^{\Delta^1}$ of a category $C$ one can associate a set of \emph{$K$-fibrations}, which have the right lifting property against elements of $K$, $K$-\emph{fibrant} objects (when $C$ has a final object), and $K$-\emph{cofibrations}, which are the maps having the left lifting property against $K$-fibrations.
  
  We write $K-\mathrm{cof}$ for the set of $K$-cofibrations, and $K-\mathrm{cell}$ for the set of (relative) $K$-cell complexes \cite[Def.~3.3]{mazelgee2015modela}, that is, the smallest set of arrows containing $K$ and stable under pushouts and transfinite compositions.
  We always have $K-\mathrm{cell}\subseteq K-\mathrm{cof}$.
  Conversely, if $K$ admits the small object argument \cite[\S10.5]{hirschhorn2009model}, every element of $K-\mathrm{cof}$ is a retract of an element of $K-\mathrm{cell}$.
  
\end{para}

\begin{definition}[Fibration] \label{bousfield/model/fibration}

  Let $L:C\rightarrow D$ be a Bousfield localisation.
  An $L$-\emph{fibration} is a map in $C$ that has the right lifting property against all $L$-quivalences.
  An object of $C$ is $L$-\emph{fibrant} if ($C$ has a terminal object and) its map to the terminal object is an $L$-fibration.
  
\end{definition}

\begin{example}

  Every map between $L$-local objects, and hence every base change thereof, is an $L$-fibration.
  
\end{example}
  
\begin{definition}[Left localisation model structure] \label{bousfield/model}

  Let $C$ be an $\infty$-category, $L:C\rightarrow D$ a Bousfield localisation.
  A model structure $(W,Cof,Fib)$ on $C$ is said to be the \emph{left localisation model structure} associated to $L$ if:
  \begin{itemize}
    \item
      $Cof = C^{\Delta^1}$;
    \item
      $W = W_L$.
  \end{itemize}
  The fibrations of the Bousfield localised model structure, if it exists, are then exactly the $L$-fibrations; the trivial fibrations are isomorphisms.
    
\end{definition}

We begin by reproving a well-known result about fibrant objects in Bousfield localized model structures (compare \cite[Prop.~3.4.1]{hirschhorn2009model}):

\begin{lemma}

  Let $C$ be cocomplete, $K$ a set of morphisms in $C$ that is stable under pushout and satisfies the 2-out-of-3 proprety.
  Then
  \[
    [f:X\rightarrow Y]\in K \quad \Rightarrow \quad [S^n\otimes_XY\rightarrow Y]\in K 
  \]
  for any $n:\N$, where $-\otimes_X-:\Space\times X\slice C\rightarrow X\slice C$ denotes the tensoring.
  
\end{lemma}
\begin{proof}

  By induction on $n$: if $X'\defeq S^n\otimes_XY\rightarrow Y$ is in $K$ then so is $Y\rightarrow Y\sqcup_{X'} Y$ by pushout, hence also $S^{n+1}\otimes_XY=Y\sqcup_{X'}Y\rightarrow Y$ by 2-out-of-3. \qedhere

\end{proof}

\begin{proposition} \label{bousfield/fibration/orthogonality}

  Let $K$ be as above. 
  Then every $K$-fibration is right orthogonal to $K$.
  
\end{proposition}
\begin{proof}

  Let $X\rightarrow Y$ be a $K$-fibration, $f:U\rightarrow V$ in $K$.
  The lifting properties for $S^n\otimes_UV\rightarrow V$ tell us that 
  \[
    \Map_{U\slice C\slice Y}(V,X) \rightarrow \Map_{U\slice C \slice Y}(V,X) ^{S^n}
  \]
  is an epimorphism of spaces for every $n:\N$ (see \textbf{Conventions} for notation).\qedhere

\end{proof}  

\begin{corollary}

  Every $K$-fibration with $K$-local target has $K$-local source.

\end{corollary}

\begin{corollary} \label{bousfield/fibration/weak-equivalence-implies-iso}

  Let $X\stackrel{f}{\rightarrow}Y\stackrel{g}{\rightarrow}Z$ be a string of maps, and suppose that $g$ and $gf$ are fibrations and $f$ is a weak equivalence.
  Then $f$ is an isomorphism.
  
\end{corollary}
\begin{proof}

  By \cite[Prop.~5.2.8.6.(3)], $f$ is right orthogonal to itself, hence invertible. \qedhere
  
\end{proof}

In \cite[Ex.~2.12]{mazelgee2015modela}, the author asks if there are criteria that guarantee the existence of left localisation model structure.
In fact, as in classical model category theory (and by the same argument), they exist quite generally:

\begin{theorem}[Existence of Bousfield localisation for combinatorial model $\infty$-categories] \label{bousfield/model/exists}

  Let $L:C\rightarrow D$ be a left Bousfield localisation of a presentable $\infty$-category.
  Then $C$ admits a left Bousfield localisation model structure which presents $L$.
  
\end{theorem}
\begin{proof}

  Let $\kappa$ be a regular cardinal such that the inclusion $R$ of $D$ into $C$ is $\kappa$-accessible. 
  We set $I\defeq (C^\kappa)^{\Delta^1}$ (where $C^\kappa\subset C$ is the category of $\kappa$-compact objects) and $J\defeq W_L\cap I \subseteq I$.
  After \cite[Thm.~3.11]{mazelgee2015modela}, we must show that $(I-\mathrm{cof}) \cap W_L\subseteq (J-\mathrm{cof}) $ (the first two conditions being obvious).
  By the small object argument \cite[\S10.5]{hirschhorn2009model}, it is enough to show that $(I-\mathrm{cell})\cap W_L\subseteq (J-\mathrm{cell})$.
  This follows from the fact that $W_L$ is closed under pushouts and transfinite composition. \qedhere
  
\end{proof}

\begin{remark}

  I chose to define only the case of Bousfield localisation relevant to our examples.
  However, the arguments apply more generally to the case where $C$ already has a non-trivial model structure: Proposition \ref{bousfield/fibration/orthogonality} works as long as trivial cofibrations are closed under 2-out-of-3 (which is automatic if all maps are cofibrations), and Theorem \ref{bousfield/model/exists} applies as long as weak equivalences are closed under colimits.
  
\end{remark}  



\section{Locally Cartesian} \label{locally-cartesian/}

This section is devoted to obtaining useful characterisations of locally Cartesian localisations.

\begin{proposition}
\label{locally-cartesian/criterion}

  Let $L:C\rightarrow D$ be a left Bousfield localisation. Then the conditions \ref{locally-cartesian/criterion/cartesian-over-local}, \ref{locally-cartesian/criterion/right-proper}, and \ref{locally-cartesian/criterion/homotopy-pullback} are equivalent, and collectively imply \ref{locally-cartesian/criterion/cartesian-unit}:
  \begin{enumerate}
    
    \item \label{locally-cartesian/criterion/cartesian-over-local}
      For any $X:C$ and $S\rightarrow LX$ in $D$ the induced map $S\times_{LX}X\rightarrow S$ is an $L$-equivalence.
      
    \item \label{locally-cartesian/criterion/right-proper}
      $C$ admits a right proper model structure presenting $L$.

    \item \label{locally-cartesian/criterion/homotopy-pullback}
      $L$ preserves pullbacks of $L$-fibrations.
      
    \item \label{locally-cartesian/criterion/cartesian-unit}
      a morphism $f:X\rightarrow Y$ is an $L$-fibration if and only if the square
      \[
        \begin{tikzcd}
          X \ar[r] \ar[d] & LX \ar[d] \\
          Y \ar[r] & LY
        \end{tikzcd}
      \]
      is Cartesian.
  \end{enumerate}
  
\end{proposition}
\begin{proof}

  \begin{labelitems}
  
    \item[\ref{locally-cartesian/criterion/cartesian-over-local}$\Rightarrow$\ref{locally-cartesian/criterion/cartesian-unit}]:
      Form the diagram
      \[
        \begin{tikzcd}
          X \ar[rrd, bend left, "W"] \ar[rd, "W" description] \ar[ddr, bend right, "F"'] \\
          & Y\times_{LY}LX \ar[r, "W"] \ar[d, "F"] & LX \ar[d, "F"] \\
          & Y \ar[r, "W"] & LY
        \end{tikzcd}
      \]
      where the arrow $Y\times_{LY}LX \rightarrow LX$, and therefore $h:X\rightarrow Y\times_{LY}LX$, is an $L$-equivalence.
      By Corollary \ref{bousfield/fibration/weak-equivalence-implies-iso}, $h$ is an isomorphism.
      
    \item[\ref{locally-cartesian/criterion/cartesian-over-local}+\ref{locally-cartesian/criterion/cartesian-unit}$\Rightarrow$\ref{locally-cartesian/criterion/right-proper}]
      By composing the Cartesian squares 
      \[
        \begin{tikzcd}
          X\times_ZY \ar[r] \ar[d] & X \ar[d, "F_L"] \ar[r] & LX \ar[d] \\
          Z \ar[r, "W_L"] & Y \ar[r] & LY 
        \end{tikzcd}
      \]
      (the second of which is Cartesian by \ref{locally-cartesian/criterion/cartesian-unit}), we find $X\times_YZ \cong X\times_{LY}LZ$.
      Apply \ref{locally-cartesian/criterion/cartesian-over-local}.  

    \item[\ref{locally-cartesian/criterion/right-proper}$\Rightarrow$\ref{locally-cartesian/criterion/cartesian-over-local}]
      Apply right propriety to the pullback of the fibration $S\rightarrow LX$ along the weak equivalence $X\rightarrow LX$.
      
    \item[\ref{locally-cartesian/criterion/right-proper}+\ref{locally-cartesian/criterion/cartesian-unit}$\Rightarrow$\ref{locally-cartesian/criterion/homotopy-pullback}]
      Let $X\rightarrow Y$ be a fibration. Applying \ref{locally-cartesian/criterion/cartesian-unit} to the diagram
      \[
        \begin{tikzcd}
          X\times_YZ \ar[r] \ar[d] & X \ar[r, "W_L"] \ar[d, "F_L"] & LX \ar[d, "F_L"] \\
          Z \ar[r] & Y \ar[r, "W_L"] & LY
        \end{tikzcd}
      \]
      we obtain that $X\times_YZ \cong LX\times_{LY}Z$.
      But then, applying \ref{locally-cartesian/criterion/right-proper} to the Cartesian squares
      \[
        \begin{tikzcd}
          Z\times_{LY}LX \ar[r, "W_L"] \ar[d, "F_L"] & LZ\times_{LY}LX \ar[r] \ar[d, "F_L"] & LX \ar[d, "F_L"] \\
          Z \ar[r, "W_L"] & LZ \ar[r] & LY
        \end{tikzcd}
      \]
      we obtain that $LX\times_{LY}LZ\cong L(LX\times_{LY}Z)\cong L(X\times_ZY)$ because both are local and weakly equivalent to $LX\times_{LY}Z$.      

    \item[\ref{locally-cartesian/criterion/homotopy-pullback}$\Rightarrow$\ref{locally-cartesian/criterion/cartesian-over-local}]
      Apply \ref{locally-cartesian/criterion/homotopy-pullback} to the pullback $S\times_{LX}X$.
    \qedhere
      
  \end{labelitems}

\end{proof}

\begin{definition}[Locally Cartesian localisation]

  A Bousfield localisation $L:C\rightarrow D$ is said to be \emph{locally Cartesian} if it satisfies the equivalent conditions of Proposition \ref{locally-cartesian/criterion}.

\end{definition}

The following argument is adapted from the proof of \cite[Thm.~1.5.4]{cisinski2006prefaisceaux}:
  
\begin{proposition}[Recognition of locally Cartesian localisations] \label{locally-cartesian/generator}

  Let $L:C\rightarrow D$ be a left Bousfield localisation of a presentable category generated (as a Bousfield localisation) by $S\subseteq C^{\Delta^1}$.
  Suppose that every morphism $p:A\rightarrow B$ of $D$ is exponentiable in $C$, meaning that pullback $C\slice B\rightarrow C\slice A$ preserves colimits.
  Suppose further that for each diagram
  \[
    \begin{tikzcd}
      && A \ar[d, "p"] \\
      X \ar[r, "f"] & Y \ar[r] & B
    \end{tikzcd}
  \]
  with $A,B\in D$ $f\in S$, the pullback $f\times_BA:X\times_BA\rightarrow Y\times_BA$ is an $L$-equivalence.
  Then $L$ is locally Cartesian.
  
\end{proposition}
\begin{proof}

  Let us say that a morphism $f:X\rightarrow Y$ in $C$ is $L$-\emph{acyclic} if for any map $p:A\rightarrow B$ in $D$ and $Y\rightarrow B$, the pullback $f\times_BA:X\times_BA\rightarrow Y\times_BA$ is an $L$-equivalence.
  Write $L-\mathrm{acy}$ for the set of $L$-acyclic maps; we have $S\subseteq L-\mathrm{acy}$ by hypothesis and $ L-\mathrm{acy}\subseteq W$ by applying the criterion to $p=\id_A$ for all local objects $A$.
  Lemma \ref{locally-cartesian/acyclic/saturated} below shows that $L-\mathrm{acy}$ is strongly saturated; thus $L-\mathrm{acy}=W$ and we are done. \qedhere
  
\end{proof}

\begin{lemma} \label{locally-cartesian/acyclic/saturated}

  The set of $L$-acyclic morphisms is strongly saturated.
  
\end{lemma}
\begin{proof}

  We will show that $L-\mathrm{acy}$ satisfies the 2-out-of-3 rule and is stable under pushouts and colimits in $C^{\Delta^1}$, then apply \cite[Prop.~5.5.4.15.(4)]{HTT}.
  \begin{labelitems}
  
    \item[pushout]
      Let 
      \[
        \begin{tikzcd}
          X \ar[r, "f"] \ar[d] \ar[r] &  Y \ar[d] \\
          X' \ar[r, "{f'}"] & Y'
        \end{tikzcd}
      \]
      be a pushout square in $C$ and suppose $f\in L-\mathrm{acy}$.
      Let $Y'\rightarrow B$ be any map.
      By universality of colimits, 
      \[
        \begin{tikzcd}
          X\times_BA \ar[r, "f\times_BA"] \ar[d] &  Y\times_BA \ar[d] \\
          X'\times_BA \ar[r] & Y'\times_BA
        \end{tikzcd}
      \]
      is a pushout, and by hypothesis $f\times_BA\in W$, whence $f'\times_BA$ since $W$ is stable under pushout.
      
    \item[2-out-of-3]
      Let $f:X\rightarrow Y$ and $g:Y\rightarrow Z$ with $\{g, gf\}\subseteq L-\mathrm{acy}$.
      Since $g\in W$, we may extend any map $Y\rightarrow B$ with $B\in D$ uniquely to a map $Z\rightarrow B$.
      Now pull back and apply 2-out-of-3 for $W$.
      The other two 2-out-of-3 rules follow similarly (except that no extension step is needed).
      
    \item[colimits]
      By universality of colimits and the closure of $W$.
      \qedhere
    
  \end{labelitems}
    
\end{proof}

\begin{corollary} \label{locally-cartesian/generator/strong}

  In the situation of Proposition \ref{locally-cartesian/generator}, suppose that pullbacks of elements of $S$ are in $S$.
  Then $S$ generates a locally Cartesian left Bousfield localisation.
  
\end{corollary}

\begin{remark}

  Compare \cite[Prop.~3.4.(2)]{Hoyois_2017}, which obtains this result for localisations of presheaf categories.
  The argument there relies on a specific formula for the localisation functor valid for presheaves, and so does not generalise to our setting.
  
\end{remark}

\begin{remark}

  Note that the generating set $S$ appearing in the statement of Proposition \ref{locally-cartesian/generator} usually does not generate $W_L$ as a set of trivial cofibrations: we must also close up under 2-out-of-3.
  See Example \ref{thomason/generator}.
  
\end{remark}

For a converse to the preceding statements, we have:

\begin{proposition} \label{locally-cartesian/unit-generates}

  Let $L:C\rightarrow D$ be a locally Cartesian left Bousfield localisation.
  Let $\kappa$ be a regular cardinal such that $D\subseteq C$ is $\kappa$-accessible.
  Then the set of arrows of the form $e_X:X\rightarrow RLX$ for $X:C^\kappa$ is stable for base change and generates $L$ as a Bousfield localisation. 
  
\end{proposition}
\begin{proof}
  
  By Lemma \ref{bousfield/generation} and criterion \ref{locally-cartesian/criterion/cartesian-unit} of Proposition \ref{locally-cartesian/criterion}. \qedhere

\end{proof}


One does not have to go very far to find interesting examples of localisations that are locally Cartesian but not left exact.

\begin{definition} \label{nullification}

  Let $C$ be a presentable category, $A\subseteq C$ a small set of objects. 
  The \emph{nullification} of $A$ is the Bousfield localisation of $C$ generated by the product projections $\pi_Y:X\times Y\rightarrow Y$ for $X\in A$, $Y:C$.
  By Corollary \ref{locally-cartesian/generator/strong}, a nullification functor such that all morphisms between local objects are exponentiable is locally Cartesian.
  
\end{definition}


\section{Classification of fibrations}

We now turn to a description of the \emph{category of fibrations} $\mathrm{Fib}\slice X$ over a fixed object $X:C$, which is defined to be the full subcategory of the slice $C\slice X$ whose objects are $L$-fibrations.
Throughout this section, we fix a locally Cartesian Bousfield localisation $L:C\rightarrow D$.

\begin{corollary}[to Proposition \ref{locally-cartesian/criterion}] \label{locally-cartesian/fibration/classification}

  Pullback along $X\rightarrow LX$ induces an equivalence of categories $\mathrm{Fib}\slice X\cong D\slice LX$.
  
\end{corollary}

\begin{lemma}[Five lemma in a topos] \label{five}

  Let 
  \[
    \begin{tikzcd}
      (X,x) \ar[r] \ar[d, equals] &  (Y,y) \ar[r] \ar[d, "\phi"] &  (Z,z) \ar[d, equals] \\
      (X,x) \ar[r] & (Y',y') \ar[r] &(Z,z)
    \end{tikzcd}
  \]
  be a map of fibre sequences in an $\infty$-topos $C$, where $Z$ is connected.
  Then $\phi$ is an isomorphism.
  
\end{lemma}
\begin{proof}

  By \cite[Prop.~6.5.1.20]{HTT}, $z:*\rightarrow Z$ is an effective epimorphism,\footnote{Thanks are due to Marc Hoyois for pointing this out to me.} whence by descent theory $C\slice Z$ is the limit of the Cech nerve of $C\slice z$.
  Now, the map from a limit over the simplex category into its zeroth term is always conservative --- as can be seen, for example, by realising the limit as Cartesian sections of the associated Cartesian fibration --- so in particular pullback along $z$ is conservative. \qedhere

\end{proof}
  
\begin{proposition} \label{koszul/thick}
  
  Let $C$ be an $\infty$-topos, $D\subseteq C$ a locally Cartesian localisation.
  Then for any fibre sequence
  \[
    (X,x)\rightarrow (Y,y) \rightarrow (Z,z)
  \]
  of pointed objects of $C$ with $Z$ connected, $(X$, $Z\in D) \Rightarrow (Y\in D)$.
  
\end{proposition}
\begin{proof}

  By the five lemma \ref{five}. \qedhere

\end{proof}
\begin{theorem}[Classification of fibrations] \label{koszul/classification}

  Let $L:C\rightarrow D$ be a locally Cartesian localisation, and let $(X,x)$ be a pointed object of $C$.
  Suppose that one of the following conditions are satisfied:
  \begin{enumerate}
    \item
      $C$ is an $\infty$-topos and $LX$ is connected as an object of $C$. 
    \item
      $D$ is an $\infty$-topos and $LX$ is connected as an object of $D$.
  \end{enumerate}
  Then a morphism $f:Y\rightarrow X$ is an $L$-fibration if and only if the induced map $\ker(f) \rightarrow \ker(Lf)$ is an equivalence. 
  Moreover, Koszul dualty induces an equivalence of categories
  \[
    \mathrm{Fib}\slice X \cong \Mod_{\Omega_xLX}(D)
  \]
  that intertwines the fibre functor $x^*:C\slice X\rightarrow C$ with the underlying object functor $\mathrm{fgt}:\Mod_{\Omega_xLX}(D)\rightarrow D\subseteq C$.
  
\end{theorem}
\begin{proof}

  By \cite[Thm.~3.1]{beardsley2021koszul} and Corollary \ref{locally-cartesian/fibration/classification}. 
  For the second case, the application is immediate.
  We reduce the first to the case of the ambient topos as follows:
  \begin{itemize}
    \item
      For a group-like monoid $A$ in $D$, $\Mod_A(D)\subseteq \Mod_A(C)$ is the full subcategory of modules whose underlying object belongs to $D$. 
      
    \item
      By Proposition \ref{koszul/thick}, the domain of a morphism $S\rightarrow LX$ belongs to $D$ if and only if its fibre $x^*S$ is in $D$. \qedhere
  
  \end{itemize}
  
\end{proof}

We end this section with a method to construct examples of fibrations.

\begin{para}[Fibrations from univalent families]

  By \cite[Thm.~3.12]{gepner2017univalence}, if $L:C\rightarrow D$ is a locally Cartesian localisation of an $\infty$-topos and $f:E\rightarrow S$ is a \emph{univalent family} in $D$, then there is an associated local class of maps $\mathcal{O}_f$ in $C$ classified by $f$; that is, for $g\in\mathcal{O}_f$ there is a unique Cartesian square
  \[
    \begin{tikzcd}
      X \ar[r] \ar[d, "g"] & E \ar[d, "f"] \\
      Y \ar[r] & S.
    \end{tikzcd}
  \]
  In particular, every element of this class is an $L$-fibration.
  
  For example, by \cite[\S6]{gepner2017univalence} this situation arises for objects of $C$ whose automorphism object belongs to $D$.

\end{para}

\begin{proposition}
  
  Let $L:C\rightarrow D$ be a locally Cartesian localisation of an $\infty$-topos, $F:C$ an object.
  Suppose that $B\underline{\Aut}(F)$ and $F/\underline{\Aut}(F)$ are in $D$.
  Then $F\in D$.
  More generally, any $F$-fibre bundle is an $L$-fibration.
  
\end{proposition}


\section{Thomason model structure} \label{thomason/}

  A proof of existence of the left localisation model structure on the groupoid completion functor $|-|:\Cat\rightarrow\Space$ appears in \cite[Ap.~A]{mazelgee2015grothendieck}.
  The question of what are the fibrations is left unresolved.
  
  An argument based on Proposition \ref{locally-cartesian/criterion} addresses this gap, and is also substantially simpler.
  
\begin{proposition} \label{thomason/locally-cartesian}

  Groupoid completion is locally Cartesian.

\end{proposition}
\begin{proof}

  Groupoid completion is generated as a localisation by the maps $\Delta^1 \times I\rightarrow I$ (or even just with $I=\Delta^0$).
  In particular, it is a nullification functor in the sense of Definition \ref{nullification}.
  Since all maps between spaces are exponentiable fibrations \cite[Cor.~1.14]{ayala2020fibrations}, groupoid completion is locally Cartesian by Corollary \ref{locally-cartesian/generator/strong}.
  \qedhere

\end{proof}

\begin{definition}[Local system] \label{local-system/definition}

  A functor $p:E\rightarrow C$ is said to be a \emph{local system} under the two equivalent conditions:
  \begin{itemize}
    \item
      $p$ is (equivalent to) both a left and a right fibration.
      
    \item
      $p$ is a right fibration, and the associated presheaf descends to $|C|$.
      
  \end{itemize}
\end{definition}

\begin{corollary} \label{thomason/fibration}

  The fibrations in the Thomason model structure on $\Cat_\infty$ are precisely the local systems.
  
\end{corollary}
\begin{proof}

  By part (\ref{locally-cartesian/criterion/cartesian-unit}) of Proposition \ref{locally-cartesian/criterion}, a functor $f:C\rightarrow D$ is a Thomason fibration if and only if it is a pullback of $|f|:|C|\rightarrow |D|$.
  In particular, Thomason fibrations are local systems.
  Conversely, suppose $f$ is a local system; then in particular it is (equivalent to) a right fibration whose classifying presheaf $F:\P(D)$ descends to a presheaf $F'$ on $|D|$.
  As $e:D\rightarrow |D|$ is a localisation, $F'$ is a left Kan extension of $F$ along $e$.
  It follows that $\int_{|D|}F' = |C|$.
  Finally, by compatibility of Grothendieck integral with base change, 
  $C\cong \int_{|D|}F'\times_{|D|}D$ and we are done.
  \qedhere

\end{proof}

\begin{corollary}[Classification of Thomason fibrations]

  Let $C$ be a connected category, $c:C$. The category of Thomason fibrations over $C$ is equivalent to the category of $\Omega_c|C|$-module spaces.
  
\end{corollary}

\begin{remark}[Generating sets] \label{thomason/generator}

  Generating as a Bousfield localisation, or even as a localisation (i.e.~without closing under colimits) does not imply generating as a set of trivial cofibrations.
  For example, groupoid completion is generated as a localisation by the following sets of maps:
  \begin{itemize}
    \item
      Cofinal (resp.~coinitial) functors. The right orthogonal to this set comprises the right (resp.~left) fibrations.
      Note that pullback of right/left fibrations commutes with groupoid completion.     
     
    \item 
      The single map $\Delta^1\rightarrow\Delta^0$. Fibrations are functors whose fibres are groupoids.
      
  \end{itemize}
  In both cases, the fibrant objects are spaces, but not all weak equivalences are trivial cofibrations.
  It is natural to ask if the resulting factorisation systems lead to a (non-right-proper) model structure on $\Cat$.
  However, since it seems to be difficult to understand the accompanying class of cofibrations and trivial fibrations, the model category axioms are hard to check.
  
\end{remark}

\begin{remark}[Thomason structures on diagrams] \label{thomason/diagram}

  While Theorem \ref{bousfield/model/exists} does give us left localisation model structures on the parametrised version
  \[
    \colim:\Cat\slice C \rightarrow \P(C)
  \]
  of the groupoid completion functor, is is easy to see that these localisations are \emph{not} locally Cartesian. For an example, take the pullback
  \[
    \begin{tikzcd}
      \emptyset \ar[r] \ar[d] & \Delta^0 \ar[d, "0"] \\
      \Delta^0 \ar[r, "1"] & \Delta^1
    \end{tikzcd}
  \]
  in the category $\Cat\slice\Delta^1$.

\end{remark}

\section{Motivic spaces} \label{motive/}

Morel-Voevodsky $\mathbb{A}^1$-localisation $L_{\A^1}:\Sheaf(\Sm_k,\mathrm{Nis})\rightarrow \mathcal{H}(k)$ of the $\infty$-topos of Nisnevich sheaves is a nullification functor, hence locally Cartesian.
The class of $L_{\A^1}$-fibrations is closely related to the right orthogonal to the set of $\A^1$-local Nisnevich weak equivalences of simplicial presheaves, which have appeared in the literature under the name $\A^1$-\emph{cover} \cite{asok2009a1}.

\begin{lemma} \label{simplicial-model/lifting-problem/model}

  Let $C$ be a simplicial model category, and let
  \[
    \begin{tikzcd}
      U \ar[r] \ar[d, "g"] & X \ar[d, "f"] \\
      V \ar[r] & Y
    \end{tikzcd}
  \]
  be a commutative square in the associated $\infty$-category $C[W^{-1}]$.
  Then there exists a commutative square
  \[
    \begin{tikzcd}
      \tilde U \ar[r] \ar[d, "\tilde g"] & \tilde X \ar[d, "\tilde f"] \\
      \tilde V \ar[r] & \tilde Y
    \end{tikzcd}
  \]
  in $C$ with all objects bifibrant, $\tilde f$ a fibration, and $\tilde g$ a cofibration.
  
\end{lemma}
\begin{proof}

  By \cite[Prop.~4.2.4.4]{HTT}, the square admits a lift to an injectively bifibrant $C$-valued presheaf on $(\Delta^1)^2$; in particular, this lift has all four objects bifibrant and $\tilde f$ a fibration.
  Now factor $\tilde g$ as a cofibration followed by a trivial fibration $\tilde U\rightarrow\tilde{\tilde V} \rightarrow\tilde V$, and replace $\tilde V$ with $\tilde{\tilde V}$.
  \qedhere

\end{proof}

\begin{lemma} \label{simplicial-model/lifting-problem/solution}

  Given a square in $C$ as in Lemma \ref{simplicial-model/lifting-problem/model}, we have:
  \[
    \Map_{\tilde U/C/\tilde Y}(\tilde V\times\Delta^\bullet,\tilde X) \cong \Map_{U/C[W^{-1}]/Y}(V,X)
  \]
  i.e.~the simplicial function complex models the $\infty$-groupoid of lifts.
  
\end{lemma}
\begin{proof}

  The simplicial model category lifting axiom entails that the fibre product defining $\Map_{\tilde U/C/\tilde Y}(\tilde V,\tilde X)$ is a homotopy fibre product.
  \qedhere
  
\end{proof}

\begin{proposition}

  Let $f:X\rightarrow Y$ be a map between Nisnevich-fibrant simplicial presheaves.
  If $f$ is an $\A^1$-cover, then the induced map of Nisnevich $\infty$-sheaves is an $L_{\A^1}$-fibration.
  If $X$ and $Y$ are sheaves of sets --- for example, schemes --- the converse holds.
  
\end{proposition}
\begin{proof}

  Suppose that $f:X\rightarrow Y$ is an $\A^1$-cover of simplicial presheaves, and let 
  \[
    \begin{tikzcd}
      U \ar[r] \ar[d, "g"] & X \ar[d, "f"] \\
      V \ar[r] & Y
    \end{tikzcd}
  \]
  be a commutative square in $\Sheaf_\infty(\Sm, \mathrm{Nis})$, where $g$ is an $\A^1$-equivalence.
  By Lemma \ref{simplicial-model/lifting-problem/model}, there exists a lift of this square to $s\mathrm{PSh}(\Sm)$, extending $f$, in which $U$ and $V$ are Nisnevich fibrant and $g$ is monic.
  Since $f$ is an $\A^1$-cover, the extension problem has a unique solution.
  This also applies to $\Delta^n\otimes g$ for all $n$, so the simplicial set of solutions to this lifting problem is a point.
  By Lemma \ref{simplicial-model/lifting-problem/solution}, $\Map_{U\slice \mathcal{H}(k)\slice Y}(V,X)$ is contractible.
  
  Conversely, suppose that $f$ is an $L_{\A^1}$-fibration and that $X$ and $Y$ are simplicially constant.
  Then $\underline{\mathrm{Hom}}_{U/-/Y}(V,X)$ is simplicially constant and contractible, hence a point.
  \qedhere
  
\end{proof}

\begin{example}[Torsors]

  Some classes of examples of fibrations can be constructed from univalent families, some of which are considered in \cite[par.~6.6+]{gepner2017univalence}.
  For instance: fibre bundles whose fibre has no rational curves and finite automorphism group, and $\mathbb{G}_m$-torsors are all fibrations for $\mathbb{A}^1$-localisation.
  
  The latter class includes Cox quotients $T\rightarrow \A^n\setminus L \rightarrow X$ presenting smooth toric varieties.
  The localised fibre sequence therefore yields a long exact sequence
  \[
    \cdots \rightarrow \pi_{n+1}^{\A^1}(X) \rightarrow \pi_n^{\A^1}(T) \rightarrow \pi_n^{\A^1}(\A^n\setminus L) 
    \rightarrow \pi_n^{\A^1}(X) \rightarrow \pi_{n-1}^{\A^1}(T) \rightarrow\cdots
  \]
  which can be used to compute $\A^1$-homotopy groups of $X$ \cite{asok2009a1}.

\end{example}

\begin{example}[Higher torsors]

  By \cite[Thm.~7.2.2.26]{HTT}, if $A$ is a strictly $\A^1$-invariant Nisnevich sheaf of (discrete) Abelian groups, then any $A$-banded $n$-gerbe in $\Sheaf(\Sm,\mathrm{Nis})$ is an $\A^1$-fibration.
  In particular, this apples to $A=\G_m$.
  I leave it to the reader to investigate the implications of, say, the long exact sequence in $\A^1$-homotopy groups, for his favourite $\G_m$ $n$-gerbe.

\end{example}

\begin{example}[Fibrations from univalent families]

  Let $f:X\rightarrow Y$ be a flat morphism whose geometric fibres are $\A^1$-local.
  Define a stack $\mathcal{M}$ whose $S$-points are maps $g:W\rightarrow S$ whose geometric fibres have the isomorphism type of some geometric fibre of $f$.
  Then if $\mathcal{M}$ (or any stack containing this one) is $\A^1$-invariant, $f$ is a fibration.
  
  Let $k$ be a number field, and call a smooth, proper variety over $k$ \emph{hyperbolic} if it is algebraically hyperbolic (i.e.~admits no nonconstant maps from $\mathbb{P}^1$ or an Abelian variety), arithmetically hyperbolic (i.e.~has only finitely many rational points over any finitely generated extension of $k$), and of general type.
  If Lang's conjectures \cite{lang1986hyperbolic} hold, the first two of these conditions are equivalent and imply the third.
  By \cite[Thm.~0.2]{moller2006special} and \cite[Thm.~1.2]{javanpeykar2021arithmetic}, the moduli stack of all hyperbolic varieties over $k$ is $\A^1$-local, so the preceding remarks apply to show that any flat family of hyperbolic varieties is an $\A^1$-fibration (in fact, an $\A^1$-cover).
  
\end{example}

\section{Plus construction} \label{plus/}

The Quillen plus construction defines a localisation of $\Space$ generated by collapsing the class $\mathcal{A}$ of acyclic spaces \cite[1.E.5]{farjoun2006cellular}.
Being a nullification, it is locally Cartesian.
The local/fibrant objects are the spaces with pro-solvable fundamental group.
In particular, the category of spaces with pro-solvable fundamental group is locally Cartesian closed.

\begin{para}[Universal acyclic extension]

  The $\mathcal{A}$-cellularisation $\mathrm{CW}_{\mathcal{A}}(X)$ of a pointed space $X$ can be constructed by iteratively taking fibres of the `Hurewicz map'
  \[
    X \rightarrow X\wedge H\Z \rightarrow \tau_k(X\wedge H\Z)\cong K((H_k(X,\Z),k)
  \]
  where we assume $X$ is acyclic up to degree $k-1$.
  (The first stage is special: it proceeds by passing to the cover of $X$ corresponding to the perfect radical of $\pi_1(X,x)$, which is not necessarily equal to $[\pi_1(X,x),\pi_1(X,x)]$ \cite[par.~2.2]{dror1972acyclic}).
  Hence $\Omega_x\mathrm{CW}_\mathcal{A}(X)$ may be characterised as a \emph{universal acyclic extension} of the $\infty$-group $\Omega_xX$.
  In particular, if $X=BG$ for $G$ a discrete group, $\pi_0(\Omega_x\mathrm{CW}_\mathcal{A}(X))$ is the universal central extension of the perfect radical of $G$.

\end{para}

\begin{corollary}[Classification of plus-fibrations]

  The category of plus-fibrations over a pointed connected space $(X,x)$ is equivalent to the category of $\Omega_xX$-modules equipped with a trivialisation over the universal acyclic extension of $\Omega_xX$; that is,
  \[
    \mathrm{Fib}\slice X \cong \Mod_{\Omega_xX}\times_{\Mod_{\Omega_x\mathrm{CW}_\mathcal{A}(X)}}\Space.
  \]
  for any connected pointed space $(X,x)$.
  
\end{corollary}
\begin{proof}

  By \cite[Rmk.~4.2]{hausmann1979acyclic}, we have a cofibre sequence $\mathrm{CW}_\mathcal{A}(X) \rightarrow X \rightarrow LX$, and hence by the Rezk-Lurie descent theorem an equivalence
  \[
    \Space\slice LX \cong \Space\slice X \times_{\Space\slice \mathrm{CW}_\mathcal{A}(X)}\Space. \qedhere
  \]

\end{proof}

\begin{corollary}

  Suppose $f:X\rightarrow Y$ is a plus-fibration. For each $y:Y$, the perfect radical of $\pi_1(Y,y)$ acts trivially on $\pi_0(X_y)$ and $\pi_{\geq 1}(X_y,x)/\pi_1(X_y,x)$ for all $x:X_y$.
  
\end{corollary}

\begin{example}[Fibrations from univalent families]

  A map $f:X\rightarrow Y$ of connected spaces whose homotopy fibre $F$ has both $\pi_1(F)$ and $\pi_0(\Aut(F))$ pro-solvable is a $+$-fibration.
  Indeed, using the long exact sequence associated to $\Aut(F)\rightarrow F\rightarrow F/\Aut(F)$, this condition implies that both $B\Aut(F)$ and $F/\Aut(F)$ are $+$-local.
  
  For example, this applies to spherical bundles of all dimensions, and hence towers thereof, because the group of connected components of $\Aut(S^n) = \{\pm 1\}\times_{\pi_nS^n}\Omega^nS^n$ is cyclic of order two.
  Hence, the plus construction preserves pullbacks of such maps.
  
\end{example}

\begin{example}
  
  On the other hand, nontrivial torus fibrations of dimension greater than 1 are usually not $+$-fibrations: the universal $n$-torus fibration is $BGL_n(\Z) \rightarrow BT^n\ltimes BGL_n(\Z)$, and $\mathrm{SL}_n(\Z)$ acts non-trivially on $\pi_1(T^n)=\Z^n$.
  
\end{example}

\begin{example}[Plus-fibrations and algebraic $K$-theory]

  Let $\phi:A\rightarrow B$ be an algebra homomorphism.
  If $B\GL(\phi)$ is a plus-fibration, then the resulting sequence of $K$-theory spaces
  \[
    B\GL(\ker\phi)\cong K(\ker\phi) \rightarrow K(A) \rightarrow K(B)
  \]
  is a fibre sequence; in particular, the $K$-theory of the non-unital algebra $\ker\phi$ computes the \emph{relative} $K$-theory $K(A,B)$.
  One also says that $\ker\phi$ \emph{satisfies excision in (connective) algebraic $K$-theory}.
  
  On the other hand, it is known that a non-unital algebra $I$ that satisfies excision in algebraic $K_1$ also satisfies $I^2=I$ (for example, by \cite{suslin1995excision} and the fact that $\mathrm{Tor}^{\Z\ltimes I}_1(\Z,\Z)=I/I^2$).
  The standard expression for expressing elementary matrices as commutators then applies: $E_{ij}(a) = [E_{ik}(b),E_{kj}(c)]$, where $a=bc$.
  Consequently, the subgroup of $\GL(I)$ generated by elementary matrices is perfect, and in particular, $\GL(I)$ cannot be pro-solvable.
  
  In other words, an algebra homomorphism \emph{never} induces a plus-fibration on $B\GL$, and this concept is not useful in algebraic $K$-theory of discrete algebras.
  
\end{example}
\printbibliography
\end{document}